\documentclass[12pt]{article}

\usepackage{amssymb}
\usepackage{amsfonts}
\usepackage{theorem}
\usepackage{bbm}
\usepackage{graphicx}
\usepackage[leqno]{amsmath}
\usepackage{enumerate}
\usepackage{indentfirst}
\usepackage{color}

\newtheorem{thm}{Theorem}

\newtheorem{remark}{Remark}

\newtheorem{lemma}{Lemma}[section]

\newtheorem{definition}{Definition}[section]

\newenvironment{proof}[1][Proof]{\textbf{#1.} }{\hfill $\square$}

\newcommand{\al}{\alpha}
\newcommand{\eps}{\varepsilon}

\newcommand{\mA}{\mathbf{A}}

\newcommand{\mP}{\mathbf{P}}

\title{Homogenization of random Navier-Stokes-type system for
electrorheological fluid }
\author{
Andrey Piatnitski \thanks{Faculty of Technology, Narvik University College, Norway \  and \ Lebedev Physical Institute RAS, Moscow, Russia\hfill\break
e-mail: \ {\tt andrey@sci.lebedev.ru}}
\and  Vasily  Zhikov \thanks{Vladimir State University, Vladimir, Russia \hfill\break
e-mail: \ {\tt zhikov@vlsu.ru}}
}

\date{\today}

\begin{document}
\maketitle

\begin{abstract}
The paper deals with homogenization of Navier-Stokes-type system
describing electrorheologial fluid with random characteristics.
Under non-standard growth conditions we construct the homogenized model
and prove the convergence result. The structure of the limit equations is also studied.
\end{abstract}


\section{Introduction}

Rheological properties of some fluids might
change essentially in the presence of an electromagnetic
field. For such fluids the viscous stress tensor is not only a nonlinear function of the deformation
velocity tensor, it also depends on the spatial argument.
A collection of interesting experimental data as well as a number of mathematical models
of electrorheological fluids can be found in \cite{Ruz}.

In this work we assume that the driving electromagnetic field has a random statistically homogeneous
microstructure.  Then the viscous stress tensor of the fluid is getting a random rapidly oscillating function of
the spatial variables.
The corresponding system of equations takes the form (the so-called generalized Naviers-Stokes
equations)
\begin{equation}\label{n-s_ini}
\left\{
\begin{array}{l}\displaystyle
\frac{\partial u^\eps}{\partial t}-{\rm div}\Big( A\big(\frac{x}{\eps}, Du^\eps\big)\Big)+{\rm div}(u^\eps\otimes u^\eps)+\nabla \pi=0,\quad \hbox{in } G\!\times\!
(0,T),\\[3mm]
{\rm div}\,u^\eps=0,\qquad u^\eps|_{\partial G}=0,\qquad u|_{t=0}=u_0,
\end{array}
\right.
\end{equation}
where the viscous stress tensor $A(y,\xi)$ satisfies non-standard $p(\cdot)$-growth conditions which are specified in details in the next Section.
¤
In (\ref{n-s_ini}) $u^\eps$ denotes the fluid velocity field and $Du^\eps$ stands for its symmetrized gradient, $\pi$ is the pressure, ${\rm div}(u^\eps\otimes u^\eps)$ is the nonlinear convective term,
 and $ A(x, Du^\eps)$ is the viscosity stress tensor of the fluid; $\eps$ is a small positive parameter that characterizes the microscopic length scale.

The goal of this work is to study the limit behaviour of $u^\eps$ as $\eps\to 0$. We assume that  $A(y,\xi)$ is a symmetric matrix being  a random ergodic statistically homogeneous function of $y\in\mathbb R^d$. In particular, the exponent $p(y)$ that characterize the growth conditions of  $A(y,\xi)$ might be
a random statistically homogeneous function.  Under a monotonicity assumption and certain conditions on $p$, we construct the effective model and prove
the homogenization result. We show in particular that the homogenized system is deterministic.

Similar results in the periodic framework have been obtained in \cite{Zh11}.  Qualitative theory of a generalized Navier-Stokes system were developed
in  \cite{DeRuWo10} and \cite{Zh09_fa}.

Our approach relies on a priori estimates, monotonicity arguments, generalized div-curl Lemma and ergodic theorems.

 \section{Problem setup}\label{sec_pbmsetup}

Given a Lipschitz bounded domain $G$ in $\mathbb R^d$ we study  initial-boundary problem (\ref{n-s_ini})
in $Q_T=G\times [0,T]$ for a fixed $T>0$.

Let $(\Omega,\mathcal{F},{\bf P})$ be a standard probability space with a measure preserving dynamical system $\tau_y$, $y\in\mathbb R^d$.
We recall that $\tau_y$  is a group of measurable mappings $\tau_y\,:\,\Omega\mapsto\Omega$ such that
\begin{itemize}
\item $\tau_{y_1+y_2}=\tau_{y_1}\circ\tau_{y_2}$,\qquad $\tau_0=Id$.
\item ${\bf P}(\tau_y(\mathcal{Q}))={\bf P}(\mathcal{Q})$ for any $\mathcal{Q}\in\mathcal{F}$ and any $y\in\mathbb R^n$.
\item $\tau\,:\,\Omega\times\mathbb R^n\mapsto\Omega$ is measurable; we assume here that $\mathbb R^d$ is equipped with the Borel
$\sigma$-algebra.
\end{itemize}
In what follows we assume that the dynamical system $\tau_\cdot$ is ergodic that is any function which is invariant with respect to $\tau_\cdot$ is
equal to a constant almost surely (a.s.).

We also assume that $\Omega$ is a compact metric space and that $\tau$ is continuous with respect to this topology.

Now we set
$$
A(y,\xi)= {\bf A}(\tau_y\omega,\xi)
$$
where ${\bf A}={\bf A}(\omega,\xi)$ possesses the following properties:
\begin{itemize}
\item[{\bf h1.}]   ${\bf A}\,:\,\Omega\times{\bf M}\mapsto{\bf M}$, where ${\bf M}$  is the space of symmetric $d\times d$-matrices which is identified with
$\mathbb R^{\frac{d(d+1)}{2}}$. We assume that  ${\bf A}$ is a Carath\'eodory function, that is ${\bf A}$ is  continuous
in $\xi$ for almost all $\omega\in\Omega$ and measurable in $\omega$ for any $\xi$.
\item[{\bf h2.}]
For all $\omega\in\Omega$ and $\xi_1\not=\xi_2$
$$
\big({\bf A}(\omega,\xi_1)-{\bf A}(\omega,\xi_2),\xi_1-\xi_2\big)>0.
$$
\item[{\bf h3.}]
There exists $c_0>0$ such that
$$
\big({\bf A}(\omega,\xi),\xi\big)\geq c_0|\xi|^{p(\omega)}-(c_0)^{-1}.
$$
\item[{\bf h4.}]
There exists $c_1>0$ such that
$$
\big|{\bf A}(\omega,\xi)\big|^{p'(\omega)}\leq c_1|\xi|^{p(\omega)}+c_1, \qquad p'(\omega)=\frac{p(\omega)}{p(\omega)-1},
$$
where the random variable  $p(\omega)$ satisfies the following estimates:
\begin{equation}\label{exp_bouu}
1<\alpha\leq p(\omega)\leq \beta<\infty.
\end{equation}
\end{itemize}

\subsection{Functional spaces}\label{ss_func-sp}

We introduce here several functional spaces. We denote
$$
C_{0,{\rm sol}}^\infty(G)=\{\psi\in C_0^\infty(G;\mathbb R^d)\,,\,\mathrm{div} \psi=0\},
$$
and $H$ is the closure of $C_{0,{\rm sol}}^\infty(G)$ in $L^2(G\,;\,\mathbb R^d)$ norm. We also define
$X^\eps$ as the closure of the space
$ C^\infty([0,T]; C_{0,{\rm sol}}^\infty(G))$ in the Luxemburg norm
$$
\|D\psi\|_{L^{p_\eps}(Q_T)}=\inf\Big\{\lambda>0\,:\,   \int_{Q_T}\big|\lambda^{-1}D \psi\big|^{p_\eps(x)}\,dxdt\leq 1\Big\};
$$
here $Q_T=G\times(0,T)$ and $p_\eps(x)=p(\tau_{x/\eps}\omega)$. Observe that the space $X^\eps$  depends on $\omega$.

\medskip
We say that a vector function $u\in X^\eps\cap L^\infty((0,T);H)$ is a {\em weak solution} of problem (\ref{n-s_ini})
if
\begin{itemize}
\item[(i)]
for any $\varphi\in C_{0,{\rm sol}}^\infty$ and for any $t',t''\in[0,T]$ the relation holds
$$
\int_G [u(x,t'')-u(x,t')]\cdot\varphi(x)\,dx+\int_{t'}^{t''}\!\!\int_G\Big[A\big(\frac{x}{\eps},Du\big)-u\otimes u\Big]
\cdot D\varphi\,dxdt=0;
$$
\item[(ii)]
$$
\lim\limits_{t\to+0} \int_G u(x,t)\cdot\varphi(x)\,dx=\int_G u_0(x)\cdot\varphi(x)\,dx
$$
\item[(iii)]
the  energy inequality
$$
\frac{1}{2}
\int_G [u(x,t'')\cdot u(x,t'')-u(x,t')\cdot u(x,t')]\,dx+\int_{t'}^{t''}\!\!\int_G A\Big(\frac{x}{\eps},Du\Big)\cdot D u\,dxdt\leq 0
$$
holds for almost all $t',\, t''\in[0,T]$.
\end{itemize}
Notice that from the definition of a solution it follows that $(u(\cdot, t),\varphi)$ is a continuous function of $t$ for any $\varphi\in C_{0,{\rm sol}}^\infty$.
In other words, $u(\cdot,t)$  is a
weakly continuous function of $t$ with values in $H$. However, it does not imply the energy equality. The theory admits
the strict energy inequality, which means the violation of energy conservation law.

\bigskip
The following statement has been proved in \cite{Zh09_fa}.

\begin{thm}\label{t_exi}
Assume that
$$
\alpha\geq \alpha_0(d)=\max\left\{\frac{d+\sqrt{3d^2+4d}}{d+2}, \frac{3d}{d+2}\right\}
$$
and
$$
\alpha\le p(x)\le \beta<\infty.
$$
Then generalized Navier-Stokes system (\ref{n-s_ini}) has a weak solution for any $u_0\in H$.
\end{thm}

\begin{remark}{\rm
In dimension $d=3$ we have $\alpha_0(3)\in (1.84, \,1.85)$.
}
\end{remark}

The condition $\al\ge \al_0$  ensures that the convective term $u\otimes u$ can be estimated in terms of  
the viscous term. More precisely, the following statement holds.

\begin{lemma} \label{l_emb_imp}
If  $u\in X\cap L^\infty(0,T,H)$,  then
$$
|u|^2 \in L^1(0,T,L^{\alpha'}(G)).
$$
\end{lemma}

\begin{remark}{\rm
In the classical case we have  $p=\frac{3d+2}{d+2}$, see \cite{Lad, Li}.  Notice that if $\alpha=\frac{3d+2}{d+2}$  then
$$
|u|^2\in L^{\alpha'}(0,T,L^{\alpha'}(G))=L^{\alpha'}(Q_T).
$$
In this case the convective term is \emph{completely} subjected to viscous one.}
\end{remark}


Due to Theorem \ref{t_exi}, for each $\eps>0$ problem (\ref{n-s_ini}) has a solution. Our goal is to study the limit behaviour of
these solutions as $\eps\to0$.

The following sections deal with the homogenization procedure.  This procedure relies on a number of auxiliary cell problems and the corresponding
functional spaces. We introduce these spaces here.

We denote by $L^{p(\cdot)}(\Omega,\mathbb R^{d(d+1)/2})$ the space of functions defined on $\Omega$ with values in the space of $d\times d$ symmetric matrices and such that
$$
\int\limits_{\Omega}|\phi(\omega)|^{p(\omega)}\,d\mP(\omega)<\infty.
$$
This space is equipped with the corresponding Luxemburg norm
$$
\|\phi\|\big._{L^{p(\cdot)}(\Omega,\mathbb R^{d(d+1)/2})}=\inf\Big\{\lambda>0\,:\,\int_\Omega |\lambda^{-1}\phi(\omega)|^{p(\omega)}\,d\mP(\omega)\leq 1\Big\}.
$$
As an immediate consequence of the properties of dynamical system $\tau$ and the Fubini theorem we have

\begin{lemma}\label{l_orl-realiz}
Let $\phi\in L^{p(\cdot)}(\Omega,\mathbb R^{d(d+1)/2})$. Then a.s.
$\phi(\tau_x\omega)\in L_{\rm loc}^{p(\tau_x\omega)}(\mathbb R^d,\mathbb R^{d(d+1)/2})$.
Moreover,
$$
\mathbf{E}\int_S |\phi(\tau_x\omega)|^{p(\tau_x\omega)}dx=|S|\int_\Omega |\phi(\omega)|^{p(\omega)}\,d\mP(\omega)
$$
for any bounded Borel set $S\subset\mathbb R^d$.
\end{lemma}
We now denote by $\partial_i$ and $\mathcal{D}_i$ the generator of $\tau$ in the $i$-th coordinate direction and its domain in $L^{2}(\Omega)$,
respectively.   We also set $\mathcal{D}=\bigcap\limits_{i=1}^d\mathcal{D}_i$ and
$$
\mathcal{D}^\infty= \{\phi\in L^\infty(\Omega)\,:\, \partial_{i_1},\ldots\partial_{i_k}\phi\in
L^{2}(\Omega)\hbox{ for all }i_1,\ldots,i_k\}.
$$
The set $\mathcal{D}^\infty$ is dense in $L^p(\Omega)$ for any $p>1$.  The realizations of functions from $\mathcal{D}^\infty$ are a.s. smooth
functions, see \cite{JKO}.

Denote $\mathcal{G}(\Omega)$ the closure of $\{D_\omega \phi,\:\, \phi\in (\mathcal{D}^\infty)^d, \, \mathrm{div}_\omega \phi=0\}$
in $L^{p(\cdot)}(\Omega)$,
where $(D_\omega \phi)_{ij}=\frac{1}{2}(\partial_i \phi_j+\partial_j \phi_i)$, and $\mathrm{div}_\omega \phi=
\partial_1\phi_1+\ldots+\partial_d \phi_d$. We then define
$$
\mathcal{G}^\perp(\Omega)=\Big\{\theta\in L^{p'(\cdot)}(\Omega;\mathbb R^{d(d+1)/2})\,:\, \int_\Omega \theta\cdot v\,d\mP(\omega)=0\ \ \hbox{for all }v\in \mathcal{G}(\Omega) \Big\}.
$$

\section{Homogenization}\label{s_homog}

In this section  we prove a number of auxiliary statements and formulate the homogenization result.
From item (iii) of the definition of a solution to problem (\ref{n-s_ini}) it follows that for each $\eps>0$
and each $\omega\in\Omega$ we have
\begin{equation}\label{aprio}
\sup\limits_{0\leq t\leq T}\|(|u^\eps(\cdot,t)|)\|\big._{L^2(G)} ^2+\int\limits_0^t\int\limits_G|Du^\eps(x,s)|^{p_\eps(x)}\,dxds\leq
C\|(|u_0|)\|^2_{L^2(G)}
\end{equation}
with a deterministic constant $C$. We recall that $p_\eps(x)=p(\tau_{x/\eps}\omega)$.
Considering {\bf h3.}, {\bf h4.} and (\ref{exp_bouu}) we derive from (\ref{aprio})

\begin{lemma}\label{l_apri}
For each $\omega\in\Omega$ the sequence $Du^\eps$ is bounded in $L^{\alpha}(Q_T; \mathbb R^{d(d+1)/2})$, and
the sequence  $A^\eps=A(x/\eps, Du^\eps)$ is bounded in $L^{\beta'}(Q_T;\mathbb R^d)$.
\end{lemma}

Using the standard arguments (see \cite[Section 5]{Zh09_fa}), one can show that  $\{u^\eps(\cdot,t)\}$ is a family of weakly
equicontinuous  functions $[0,T]\mapsto L^2(G;\mathbb R^{d(d+1)/2})$.
Moreover, by the Aubin--Lions lemma, this family is compact in  $L^2(Q_T;\mathbb R^d)$.
This yields the following convergence result.

\begin{lemma}\label{l_wconv}
For a subsequence, as $\eps\to0$,
$$
\begin{array}{rl}
\displaystyle u^\eps(\cdot,t)\rightharpoonup u(\cdot,t) \ &\displaystyle\hbox{\rm weakly in }L^2(G;\mathbb R^{d}) \ \ \hbox{\rm for all }t\in[0,T];\\[2mm]
\displaystyle u^\eps(\cdot,t)\to u(\cdot,t) \ &\displaystyle \hbox{\rm in }L^2(G;\mathbb R^{d}) \ \ \hbox{\rm for a.a. }t\in[0,T];\\[2mm]
\displaystyle Du^\eps \rightharpoonup Du\ &\displaystyle \hbox{\rm weakly in }L^{\alpha}(Q_T;\mathbb R^{d(d+1)/2)});\\[2mm]
\displaystyle A\Big(\frac{\cdot}{\eps},Du^\eps\Big) \rightharpoonup z^0 \ &\displaystyle \hbox{\rm weakly in }L^{\beta'}(Q_T;\mathbb R^{d(d+1)/2}).
\end{array}
$$
\end{lemma}

Notice that $u=u(x,t)$ and $z^0=z^0(x,t)$ might depend on $\omega$.

Passing to the limit in the integral identity (i) we obtain
\begin{equation}\label{lim_ii}
\int_G [u(x,t'')-u(x,t')]\cdot\varphi(x)\,dx+\int_{t'}^{t''}\!\!\int_G\big[z^0-u\otimes u\big]
\cdot D\varphi\,dxdt=0
\end{equation}
for any $\varphi\in C_{0,{\rm sol}}^\infty(G)$ and for any $t',t''\in[0,T]$. The crucial step now is to determine
a relation between $z^0$ and $Du$.  To this end we consider the following auxiliary problem: given $\xi\in\mathbb R^{d(d+1)/2}$
find $v_\xi\in \mathcal{G}(\Omega)$ such that
\begin{equation}\label{cellpbm}
\int_\Omega \mA(\omega,v_\xi(\omega)+\xi)\cdot \theta(\omega)\,d\mP(\omega)=0 \quad\hbox{for any }\theta\in  \mathcal{G}(\Omega).
\end{equation}
\begin{lemma}\label{l_cell-exi}
Under assumptions {\bf h1.}--{\bf h4.} problem (\ref{cellpbm}) has a unique solution for each $\xi\in\mathbb R^{d(d+1)/2}$.
\end{lemma}
\begin{proof}[{\it Proof} {\rm relies on classical result for monotone operators}] Denote by ${\cal A}_\xi$ the operator mapping
$\mathcal{G}(\Omega)$ to $\mathcal{G}^\perp(\Omega)$ and defined by ${\cal A}_\xi[\theta](\omega)=A(\omega,\xi+\theta(\omega))$.
Due to assumption {\bf h2.} this operator is monotone. From {\bf h4.} it follows that ${\cal A}_\xi$ is bounded.   Then, from {\bf h1.} and {\bf h4.}
with the help of Lebesgue theorem one can derive that the function
\begin{equation*}
s\ \longrightarrow\ \int_\Omega \mA(\omega,\xi+\theta_1(\omega)+s\theta_2(\omega))\cdot \theta_3(\omega)\,d\mP(\omega)=0
\end{equation*}
is continuous in $s\in\mathbb R$ for any $\theta_1, \,\theta_2,\,\theta_3\in\mathcal{G}(\Omega)$.   Also, as an immediate consequence of
{\bf h3.}, we have  $\|\theta\|^{-1}(\mathcal{A}_\xi(\theta),\theta)\to\infty$, as $\|\theta\|\to\infty$.  Then, by \cite[Theorem 2.2.1]{Li} problem
(\ref{cellpbm}) has a unique solution.
\end{proof}

\bigskip
The homogenized diffusion tensor is now introduced by
$$
A^{\rm eff}(\xi)=\int_\Omega \mA(\omega,\xi+v_\xi(\omega))\,d\mP(\omega).
$$

Consider an auxiliary variational problem
\begin{equation}\label{min_f}
f(\xi)=\min\limits_{v\in \mathcal{G}(\Omega)}\ \int_\Omega \frac{|\xi+v(\omega)|^{p(\omega)}}{p(\omega)}\,d\mP(\omega).
\end{equation}
The conjugate (in the sense of Young) functional takes the form
$$
f^*(\xi)=\Big\{\int_\Omega\frac{|w|^{p'(\omega)}}{p'(\omega)}\,d\mP(\omega)\,:\, w\in \mathcal{G}^\perp(\Omega),\, \int_\Omega w\,d\mP(\omega)=\xi\Big\}
$$
Both functionals $f$ and $f^*$ are convex and even. Moreover, $f(\xi)>0$ for $\xi\not=0$, and $f^*(\xi)>0$ for $\xi\not=0$. 
\begin{lemma}\label{l_delta2}
Function $f(\xi)$ satisfies the following inequality
$$
f(\lambda\xi)\leq\left\{\begin{array}{ll}
\lambda^\alpha f(\xi),\qquad \hbox{\rm if }\lambda\leq 1\\[2mm]
\lambda^\beta f(\xi),\qquad\hbox{\rm if }\lambda\geq 1
\end{array}
\right. .
$$
\end{lemma}
\begin{proof}
Denote $w_\xi$ the function in $\mathcal{G}(\Omega)$ that provides the minimum in (\ref{min_f}). We have
$$
f(\lambda\xi)=\int_\Omega\frac{|\lambda\xi+w_{\lambda\xi}(\omega)|^{p(\omega)}}{p(\omega)}\,d\mathbf{P}
\leq\int_\Omega\frac{|\lambda\xi+\lambda w_{\xi}(\omega)|^{p(\omega)}}{p(\omega)}\,d\mathbf{P}
$$
$$
\leq\int_\Omega\lambda^{p(\omega)}\frac{|\xi+ w_{\xi}(\omega)|^{p(\omega)}}{p(\omega)}\,d\mathbf{P}.
$$
This implies the desired inequality.
\end{proof}

\bigskip
Let $L^f(Q_T)$ be the associated with $f$ Orlicz space defined as
$$
L^f(Q_T)=\Big\{\phi\in L^1(Q_T,\mathbb R^{d(d+1)/2})\,:\, \int_{Q_T} f(\phi(x))\,dx<\infty\Big\}
$$
with the norm
$$
\||\phi\|\big._{L^f}=\inf\Big\{\lambda>0\,:\, \int_{Q_T}f(\lambda^{-1}\phi)\,dx\leq 1\Big\}.
$$
We also need the following Sobolev-Orlicz spaces:
$$
\begin{array}{c}
\displaystyle
W_0^{1,f}(G)=\big\{\phi\in W^{1,1}_0(G)\,:\, \mathrm{div} \phi=0,\, f(D\phi)\in L^1(G)\big\},\\[2mm]
\displaystyle
\|\phi\|\big._{W_0^{1,f}(G)}=\|D\phi\|\big._{L^f(G)}.
\end{array}
$$
and
$$
\begin{array}{c}
\displaystyle
X^f(Q_T)=\big\{\vartheta\in L^1((0,T), W^{1,1}_0(G;\mathbb R^d))\,:\, \mathrm{div}_x \vartheta=0,\, f(D\vartheta)\in L^1(Q_T)\big\},\\[2mm]
\displaystyle
\|\vartheta\|\big._{X^{f}(Q_T)}=\|D\vartheta\|\big._{L^f(Q_T)}.
\end{array}
$$

The following statement has been proved in \cite[Proposition X.2.6]{ET} 
\begin{lemma}\label{l_dens-sol}
The space $C_{0,{\rm sol}}^\infty(G)$ is dense in $W^{1,f}_0(G)$, and the space $C^\infty([0,T],C_{0,{\rm sol}}^\infty(G) )$
is dense in $X^{f}(Q_T)$.
\end{lemma}
For star-shaped domains this result can be easily proved with the help of smoothing operators. For a generic Lipschitz
domain the proof is more involved.

The properties of homogenized diffusion tensor $A^{\rm{eff}}$ are given in the following statement.
\begin{lemma}\label{l_eff-tenz}
The homogenized tensor $A^{\rm{eff}}$ is strictly monotone and continuous. Moreover, the flux $\mathbf{A}(\xi+v_\xi(\cdot))$ is a weakly continuous function
of $\xi$ with values in $L^{p'(\cdot)}(\Omega,\mathbb R^{d(d+1)/2})$.  There exist $c_0>0$ and $c_1>0$ such that
\begin{equation}\label{f-ineq}
\begin{array}{c}
\displaystyle
A^{\rm{eff}}(\xi)\cdot\xi\geq c_0f(\xi)-c_0^{-1},\\[2mm]
\displaystyle
f^*\big(A^{\rm{eff}}(\xi)\big)\leq  c_1f(\xi)+  c_1.
\end{array}
\end{equation}
\end{lemma}
\begin{proof}
Considering problem (\ref{cellpbm}) and {\bf h3.}, we have
$$
A^{\rm eff}(\xi)\cdot\xi=\int\limits_\Omega \mathbf{A}(\omega,\xi+v_\xi(\omega))\cdot\xi\,d\mathbf{P}=\int\limits_\Omega \mathbf{A}(\omega,\xi+v_\xi(\omega))\cdot(\xi+v_\xi(\omega))\,d\mathbf{P}
$$
$$
\geq c_0\int\limits_\Omega |\xi+v_\xi(\omega)|^{p(\omega)}\,d\mathbf{P}-c_0^{-1}\geq c_0f(\xi)-c_0^{-1}.
$$
This gives the first inequality in (\ref{f-ineq}).  To justify the second one we notice that $ \mathbf{A}(\xi+v_\xi)\in \mathcal{G}^\perp(\Omega)$,
and $\int_\Omega \mathbf{A}(\xi+v_\xi)\,d\mathbf{P}=A^{\rm eff}(\xi)$. Therefore, by the definition of $f^*$,
$$
\begin{array}{c}
\displaystyle
f^*(A^{\rm eff}(\xi))\leq\int\limits_{\Omega}|\mathbf{A}(\omega,\xi+v_\xi(\omega))|^{p'(\omega)}\,d\mathbf{P}\\[2mm]
\leq \displaystyle
c_2\int\limits_{\Omega}|\mathbf{A}(\omega,\xi+v_\xi(\omega))\cdot(\omega,\xi+v_\xi(\omega))\,d\mathbf{P}+c_3\\[2mm]
=c_2A^{\rm eff}(\xi)\cdot\xi+c_3\leq c_2\big(\gamma f^*(A^{\rm eff}(\xi))+C(\gamma)f(\xi)\big)+c_3;
\end{array}
$$
here we have also used {\bf h4.}, {\bf h3.},  the Young inequality and Lemma \ref{l_delta2}.
Choosing in the last expression $\gamma=(2c_2)^{-1}$,  we obtain the second estimate in (\ref{f-ineq}).

\bigskip
Strict monotonicity of $A^{\rm eff}(\xi)$ is an immediate consequence of the strict monotonicity of $\mathbf{A}(\omega,\xi)$ and the definition of $A^{\rm eff}$.
Indeed,
$$
(A^{\rm eff}(\xi_1)-A^{\rm eff}(\xi_2))\cdot(\xi_1-\xi_2)=\int_\Omega\big(\mathbf{A}(\omega,\xi_1+v_{\xi_1}(\omega))-\mathbf{A}(\omega,\xi_2+v_{\xi_2}(\omega))\big)
\cdot(\xi_1-\xi_2)\,d\mathbf{P}
$$
$$
=\int_\Omega\big(\mathbf{A}(\omega,\xi_1+v_{\xi_1}(\omega))-\mathbf{A}(\omega,\xi_2+v_{\xi_2}(\omega))\big)
\cdot(\xi_1+v_{\xi_1}(\omega)-(\xi_2+v_{\xi_2}(\omega)))\,d\mathbf{P}>0.
$$

In order to prove weak continuity of $A(\xi+v_\xi(\cdot))$ we first show that $v_\xi(\cdot)$ is a weakly continuous in $\xi$
function with values in $L^{p(\cdot)}(\Omega,\mathbb R^d)$. To this end we consider a sequence $\xi_j$ that converges to $\xi$ and notice that, due to condition {\bf h3.},
we have $\|v_{\xi_j}\|_{L^p(\cdot)}\leq C$.  Then for a subsequence $v_{\xi_j}$ converges to some $\eta\in\mathcal{G}(\Omega)$ weakly in $L^{p(\cdot)}(\Omega,\mathbb R^d)$.  By monotonicity, for any $\zeta\in \mathcal{G}(\Omega)$ it holds
$$
\int_\Omega\mathbf{A}(\omega, \xi_j+\zeta)\cdot(v_{\xi_j}-\zeta)\,d\mathbf{P}=
\int_\Omega\big(\mathbf{A}(\omega, \xi_j+\zeta)-\mathbf{A}(\omega, \xi_j+v_{\xi_j})\big)\cdot(v_{\xi_j}-\zeta)\,d\mathbf{P}\leq 0.
$$
From {\bf h1.} and {\bf h4.} we deduce by the Lebesgue theorem that $\mathbf{A}(\omega, \xi_j+\zeta)\to \mathbf{A}(\omega, \xi+\zeta)$
strongly in $L^{p'(\cdot)}(\Omega,\mathbb R^{d(d+1)/2})$.  Passing to the limit $j\to\infty$ in the last inequality yields
$$
\int_\Omega\mathbf{A}(\omega, \xi+\zeta(\omega))\cdot(\eta-\zeta(\omega))\,d\mathbf{P}\leq 0.
$$
This implies with the help of Minty's argument that $\eta$ is a solution of problem (\ref{cellpbm}). Since a solution of  (\ref{cellpbm}) is unique,
$\eta=v_\xi$. Therefore, $v_{\xi_j}$ converges to $v_\xi$.

Denote by $z$ a weak limit (for a subsequence) of $\mathbf{A}(\cdot,\xi_j+v_{\xi_j}(\cdot))$, as $j\to\infty$. Since $z\in \mathcal{G}^\perp(\Omega)$,
$$
\int_\Omega z\cdot v_\xi\,d\mathbf{P}=0,\qquad  \int_\Omega z\cdot \zeta\,d\mathbf{P}=0\quad\hbox{for all }\zeta\in \mathcal{G}(\Omega).
$$
By monotonicity,
$$
\int_\Omega \big(\mathbf{A}(\omega, \xi_j+v_{\xi_j}(\omega))-\mathbf{A}(\omega, \xi_j+\zeta(\omega))\big)\cdot(v_{\xi_j}(\omega)-\zeta(\omega))       \,d\mathbf{P}\geq0
$$
Passing to the limit $j\to\infty$ we get
$$
\int_\Omega \big(z-\mathbf{A}(\omega, \xi+\zeta(\omega))\big)\cdot(v_\xi(\omega)-\zeta(\omega))       \,d\mathbf{P}\geq0
$$
Using one more time Minty's technique we conclude that $z=\mathbf{A}(\omega, \xi+v_\xi(\omega))$.

\end{proof}


\bigskip
The homogenized problem reads
\begin{equation}\label{pbm_eff}
\left\{
\begin{array}{l}\displaystyle
\frac{\partial u}{\partial t}-{\rm div}\big( A^{\rm{eff}}(Du)\big)+{\rm div}(u\otimes u)+\nabla \pi=0,\quad (x,t)\in Q_T,\\[3mm]
{\rm div}\,u=0,\qquad u|_{\partial G}=0,\qquad u|_{t=0}=u_0,
\end{array}
\right.
\end{equation}
We say that a vector function $u\in X^f(Q_T)\cap L^\infty((0,T), H)$ is a solution of problem (\ref{pbm_eff}) if
\begin{itemize}
\item[(i)]
for any $\varphi\in C_{0,{\rm sol}}^\infty(G)$ and for any $t',t''\in[0,T]$ it holds
$$
\int_G [u(x,t'')-u(x,t')]\cdot\varphi(x)\,dx+\int_{t'}^{t''}\!\int_G\big[A^{\rm{eff}}(Du)-u\otimes u\big]
\cdot D\varphi\,dxdt=0;
$$
\item[(ii)]
$$
\lim\limits_{t\to+0} \int_G u(x,t)\cdot\varphi(x)\,dx=\int_G u_0(x)\cdot\varphi(x)\,dx
$$
\item[(iii)]
the  inequality
$$
\frac{1}{2}
\int_G [u(x,t'')\cdot u(x,t'')-u(x,t')\cdot u(x,t')]\,dx+\int_{t'}^{t''}\!\!\int_G A^{\rm{eff}}(Du)\cdot D u\,dxdt\leq 0
$$
holds for almost all $t',\, t''\in[0,T]$.
\end{itemize}

We proceed with the main homogenization result of this work.

\begin{thm}\label{t_main-homo}
Assume that
$$
\beta <\alpha^*=\left\{\begin{array}{ll}
\displaystyle
\frac{\alpha d}{d-\alpha},\qquad &\hbox{if }\alpha<d,\\[3mm]
+\infty,\qquad &\hbox{if }\alpha\geq d
\end{array}\right.
$$
Then almost surely, as $\eps\to 0$,  any limit point $u$ of the family $u^\eps$ is a solution of the homogenized problem (\ref{pbm_eff}) .
\end{thm}

\begin{remark}\label{r_rand}
Notice that the previous theorem does not state that the limit function is deterministic. Although the limit problem is not random, a solution need not be
unique. Then, the limit points of $u^\eps$ might be distinct for different realizations.
\end{remark}

\section{Stochastic two-scale convergence}\label{s_2scale-conv}

We first recall the definition of stochastic two-scale convergence.  Let $\{v^\eps=v^\eps(x,t,\widetilde\omega), \ 0<\eps\leq\eps_0\}$ be a family of functions such that
for ${\bf P}$ almost all  $\widetilde\omega\in\Omega$ we have  $v^\eps(\cdot,\cdot,\widetilde\omega)\in L^p(Q_T)$ for all $\eps\in (0,\eps_0]$.
\begin{definition}\label{d_stoch-2s}
We say that the family $v^\eps\in L^p(Q_T)$ weakly stochastic two-scale converges, as $\eps\to0$, to a function $v=v(x,t,\omega)$, $v\in L^p(Q_T\times\Omega)$,
if a.s.
\begin{equation}\label{Lp_bou}
\limsup\limits_{\eps\to0}\|v_\eps\|_{L^p(Q_T)}<\infty,
\end{equation}
and for any $\varphi\in C_0^\infty(Q_T)\times\mathcal{D}^\infty(\Omega)$ it holds
$$
\lim\limits_{\eps\to0}\int_{Q_T}v^\eps(x,t) \varphi^\eps(x,t)\,dxdt\longrightarrow
\int_{Q_T}\int_\Omega v(x,t,\omega)\varphi(x,t,\omega)\,dxdtd\mP,
$$
where $\varphi^\eps(x,t)=\varphi(x,t,\tau_{x/\eps}\omega)$.
\end{definition}

Notice that the two-scale limit function might also depend on the realization of the medium $\widetilde\omega$. Observe also that although the two-scale limit is defined separately for each typical realization of the medium, that is for a given $\widetilde\omega$, the limit
function is defined on the whole $\Omega$. We do not indicate the dependence on
$\widetilde\omega$  explicitly.

\bigskip\noindent
We recall some of the main properties of stochastic two-scale convergence (see \cite{Zh_Pi})
that are used in the further analysis.
For the reader convenience we  provide a proof of  these statements.

\begin{lemma}\label{l_2scomp} Every family of functions $\{v^\eps, \ \eps>0\}$ such that (\ref{Lp_bou}) holds, weakly two-scale converges  for a
subsequence to some $v=v(x,t,\omega)$,   $v\in L^p(Q_T\times\Omega)$.
\end{lemma}
\begin{proof}
With the help of the Birkhoff ergodic theorem we obtain that for any $\varphi\in C_0^\infty(Q_T)$, $\phi\in\mathcal{D}^\infty(\Omega)$
and for almost all $\widetilde\omega\in\Omega$
$$
\limsup\limits_{\eps\to
0}\left|\int\limits_{Q_T}v^\eps(x)\varphi(x)
\phi(\tau_\frac{x}{\eps}\tilde\omega)dx
\right|\le
$$
$$
\le\limsup\limits_{\eps\to 0}\|v^\eps\|_{L^p(Q_T)}
\left(\int\limits_{Q_T}|\varphi(x)|^q\,|\phi(\tau_\frac{x}{\eps}\tilde\omega)|^q
dx\right)^\frac{1}{q}\le
$$
$$
\le C_{\tilde\omega}\lim\limits_{\eps\to
0}\left(\int\limits_{Q_T}|\varphi(x)|^q\phi(\tau_\frac{x}{\eps}\tilde\omega)|^q
dx\right)^\frac{1}{2}=C_{\tilde\omega}
\left(\int\limits_{Q_T}\int\limits_\Omega |\varphi(x)|^q
|\phi(\omega)|^q d\mathbf{P}(\omega)dx\right)^\frac{1}{q}\!\!\!.
$$
Using the diagonal procedure we can choose a subsequence $\eps_j\to0$
such that the limit $\lim\limits_{\eps_j\to
0}\int\limits_{Q_T}v^\eps(x)\varphi(x)
\phi(\tau_\frac{x}{\eps}\tilde\omega)dx$ exists for each $\varphi$ and $\phi$.
It immediately follows from the last formula that this limit defines a linear
bounded functional on $L^q(Q_T\times\Omega)$. Therefore, there exists a function
$v\in L^p(Q_T\times\Omega)$ such that
$$
\lim\limits_{\eps\to
0}\int\limits_{Q_T}v^\eps(x)\varphi(x)
\phi(\tau_\frac{x}{\eps}\tilde\omega)dx=\int\limits_{Q_T}\int\limits_\Omega
v(x,t,\omega)\varphi(x)\phi(\omega)dxd{\bf P}.
$$
  By the density arguments the last relation also holds for any test function
$\varphi\in C_0^\infty(Q_T)\times\mathcal{D}^\infty(\Omega)$. This completes the proof.
\end{proof}

\begin{lemma}\label{l_ngue1} Let a family $v^\eps$ be such that a.s.
$$
\|v^\eps\|_{L^p(Q_T)}\leq C,\qquad \lim\limits_{\eps\to0}\ \eps\|\nabla_x v^\eps\|_{L^p(Q_T)}=0.
$$
Then, for a subsequence,
$$
v^\eps \mathop{\rightharpoonup}\limits^{2} v \qquad\hbox{weakly two-scale in }L^p(Q_T),
$$
with $v=v(x,t)$, $v\in L^p(Q_T)$.
\end{lemma}
\begin{proof}
Choosing  a test function of the form
$\varphi(x,t)\phi(\tau_{x/\eps}\omega)$, we get for a subsequence
$$
0=\lim\limits_{\eps\to0}\int\limits_{Q_T} \eps\nabla_xv^\eps(x,t)\varphi(x,t)\phi(\tau_{x/\eps}\omega)\,dx
=-\lim\limits_{\eps\to0}\int\limits_{Q_T} v^\eps(x,t)\varphi(x,t)\mathrm{div}_\omega\phi(\tau_{x/\eps}\omega)\,dx
$$
$$
=-\int_{Q_T}\int_\Omega v(x,t,\omega)\varphi(x,t)\mathrm{div}_\omega\phi(\omega)
dxd{\bf P}.
$$
Therefore, for almost all $(x,t)\in Q_T$ we have
$$
\int_\Omega v(x,t,\omega)\mathrm{div}_\omega\phi(\omega)
d{\bf P}.
$$
in the same way as in \cite[Lemma 2.5]{Zh_Pi} one can show that the set $\{\mathrm{div}_\omega\phi\,:\,\phi\in\mathcal{D}^\infty\}$ is dense in the space of $L^q(\Omega)$
functions with zero average. Therefore, $v$ does not depend on $\omega$.
\end{proof}

\begin{lemma}\label{l_ngue2} Let a family $v^\eps$ satisfy a.s the estimate
$$
\|v^\eps\|_{L^p(Q_T)}+\|\nabla_x v^\eps\|_{L^p(Q_T)}\leq C
$$
for all $\eps\in(0,\eps_0]$.
Then, for a subsequence,
$$
\nabla_x v^\eps \mathop{\rightharpoonup}\limits^{2} \nabla_x v(x,t)+ v_1(x,t,\omega)\qquad\hbox{weakly two-scale
in }L^p(Q_T\times\Omega),
$$
with $v=v(x,t)$, $v\in L^p((0,T);W^{1,p}(G))$ and $v_1\in L^p(Q_T;L^p_{\rm pot}(\Omega))$, where $L^p_{\rm pot}(\Omega)$ is the closure in $L^p(\Omega)$
of the set $\{\partial_\omega u\,:\, u\in \mathcal{D}^\infty(\Omega)\}$.
\end{lemma}
\begin{proof}
According to the previous Lemma a two-scale limit of $v^\eps$ does not depend on $\omega$.
Denote by $V=V(x,t,\omega)$ the two-scale limit of $\nabla_x v^\eps$, and by $v=v(x,t)$ the two-scale limit of $v^\eps$. Since the two-scale convergence in $L^p(Q_T\times\Omega)$
implies the weak convergence in $L^p(Q_T)$, we have $v\in L^p(0,T;W^{1,p}(Q))$.  Taking
a test function $\varphi(x,t)\phi(\tau_{x/\eps}\omega)$ with $\mathrm{div}_\omega\phi=0$, we arrive at the following relation
$$
\int_{Q_T}\int_\Omega V(x,t,\omega)\varphi(x,t)\phi(\omega)dxd{\bf P}=
\lim\limits_{\eps\to0}\int_{Q_T}\nabla_xv^\eps(x,t)\varphi(x,t)\phi(\tau_{x/\eps}\omega)dx
$$
$$
=-\int_{Q_T}\int_\Omega v(x,t)\nabla_x\varphi(x,t)\phi(\omega)dxd{\bf P}=
\int_{Q_T}\int_\Omega \nabla_x v(x,t)\varphi(x,t)\phi(\omega)dxd{\bf P}.
$$
Denoting $v_1(x,t,\omega)=V(x,t,\omega)-\nabla_x v(x,t)$ we conclude that
for almost all $(x,t)\in Q_T$ and for any $\phi\in\mathcal{D}^\infty$ such that
$\mathrm{div}_\omega\phi=0$ it holds
$$
\int_\Omega v_1(x,t,\omega)\phi(\omega)d{\bf P}=0.
$$
This implies the desired statement.
\end{proof}

\bigskip\noindent
{\sl Example. Periodic case}.

\medskip\noindent
The periodic framework can be interpreted as a particular case of the random one. In this case $\Omega=[0,1)^d$, $\mathcal{F}$
is the Borel $\sigma$-algebra on $\Omega$, and $\mathbf{P}$ is the Lebesgue measure.   The dynamical system $\tau_y$
is the set of shifts on the torus, that is for any $\omega\in [0,1)^d$ we set $\tau_y\omega=\mathcal{I}(\omega+y)$,
where $\mathcal{I}(\omega+y)\in[0,1)^d$, and $(\omega+y)-\mathcal{I}(\omega+y)\in\mathbb Z^d$.
One can observe that in the periodic case for any $\omega_1$ and $\omega_2$ there exist $y\in\mathbb Z^d$ such that
$\omega_2=\tau_y\omega_1$. This property plays a crucial role in the analysis of periodic media.

In the periodic case
 Lemmas \ref{l_2scomp}--\ref{l_ngue2} are classical and can be found in \cite{nguetseng}, \cite{allaire}.

\bigskip
Considering a priori estimate (\ref{aprio}) and using the arguments from \cite{Zh_Pi} and \cite{Zh11,Zh09_fa}, one can justify the following statement:
\begin{lemma}\label{l_sol-2sca} For a subsequence,
$$
 u^\eps \mathop{\rightharpoonup}\limits^{2} u(x,t)\qquad\hbox{weakly two-scale in }L^{\alpha}(Q_T),
$$
$$
D u^\eps \mathop{\rightharpoonup}\limits^{2}D u(x,t)+ u_1(x,t,\omega)\qquad\hbox{weakly two-scale in }L^{\alpha}(Q_T\times\Omega),
$$
where $u_1(x,t,\cdot)\in \mathcal{G}(\Omega)$ a.a. in $Q_T$ and
\begin{equation}\label{vari_lsc}
\int_{Q_T}\int_{\Omega}|Du(x,t)+u_1(x,t,\omega)|^{p(\omega)}\,dxdtd\mP(\omega)<\infty;
\end{equation}
$$
A(\cdot/\eps,Du^\eps) \mathop{\rightharpoonup}\limits^{2} z(x,t,\omega)\qquad\hbox{weakly two-scale in }L^{\beta'}(Q_T\times\Omega),
$$
where
$$
\int_{Q_T}\int_\Omega |z(x,t,\omega)|^{p'(\omega)}\,dxdtd\mP(\omega)<\infty,
$$
$z(x,t,\cdot)\in \mathcal{G}^\perp(\Omega)$  a.a. in $Q_T$. Moreover,
$z_0(x,t)=\int_\Omega z(x,t,\omega)\,d\mP(\omega)$ with $z_0$ introduced in Lemma \ref{l_wconv}.
\end{lemma}

\begin{proof}
The two-scale convergence follows from the previous Lemmas. We should justify (\ref{vari_lsc}) and similar estimate for $z$.
Denote for brevity $U(x,t,\omega)=D u(x,t)+ u_1(x,t,\omega)$. For any $\gamma>0$ consider
$U^\gamma\in C_0^\infty(Q_T)\times\mathcal{D}^\infty(\Omega)$ such that
$\|U-U^\gamma\|_{L^{\alpha}(Q_T\times\Omega)}\leq \gamma$. For any $\delta\in(0,1)$ by the convexity argument we have
\begin{equation}\label{tri_tri}
\int\limits_{Q_T}\big|(1-\delta)U^\gamma(t,x,\tau_{\frac x\eps}\omega)+
\delta Du^\eps(t,x)\big|^{p(\tau_{\frac x\eps}\omega)}dxdt
\end{equation}
$$
\leq (1-\delta)\int\limits_{Q_T} |U^\gamma(t,x,\tau_{\frac x\eps}\omega)|^{p(\tau_{\frac x\eps}\omega)}
dxdt +\delta\int\limits_{Q_T} |Du^\eps(t,x)|^{p(\tau_{\frac x\eps}\omega)}
dxdt.
$$
Using the inequality $|a+\delta b|^p-|a|^p-\delta p|a|^{p-2}ab=o(\delta)\,(|a|^p+|b|^p)$, as  $\delta\to0$, that holds uniformly
in $a$ and $b$, we obtain
$$
\big|(1-\delta)U^\gamma(t,x,\tau_{\frac x\eps}\omega)+
\delta Du^\eps(t,x)\big|^{p(\tau_{\frac x\eps}\omega)}= \big(1-\delta p(\tau_{\frac x\eps}\omega)\big)
|U^\gamma(t,x,\tau_{\frac x\eps}\omega)|^{p(\tau_{\frac x\eps}\omega)}
$$
$$
+\delta p(\tau_{\frac x\eps}\omega)|U^\gamma(t,x,\tau_{\frac x\eps}\omega)|^{p(\tau_{\frac x\eps}\omega)-2}
U^\gamma(t,x,\tau_{\frac x\eps}\omega)Du^\eps(t,x)
$$
$$
+o(\delta) \big(|U^\gamma(t,x,\tau_{\frac x\eps}\omega)|^{p(\tau_{\frac x\eps}\omega)}
+|Du^\eps(t,x)|^{p(\tau_{\frac x\eps}\omega)}\big).
$$
Integrating the last equality over $Q_T$ and combining the resulting relation with  \eqref{tri_tri} after straightforward rearrangements we obtain
$$
\int\limits_{Q_T} |Du^\eps(t,x)|^{p(\tau_{\frac x\eps}\omega)}\,dxdt\geq
\int\limits_{Q_T} \big(1-p(\tau_{\frac x\eps}\omega)\big)|U^\gamma(t,x,\tau_{\frac x\eps}\omega)|^{p(\tau_{\frac x\eps}\omega)}
dxdt
$$
$$
+\int\limits_{Q_T} p(\tau_{\frac x\eps}\omega)|U^\gamma(t,x,\tau_{\frac x\eps}\omega)|^{p(\tau_{\frac x\eps}\omega)-2}
U^\gamma(t,x,\tau_{\frac x\eps}\omega)Du^\eps(t,x)\,dxdt
$$
$$
+o_\delta(1) \big(|U^\gamma(t,x,\tau_{\frac x\eps}\omega)|^{p(\tau_{\frac x\eps}\omega)}
+|Du^\eps(t,x)|^{p(\tau_{\frac x\eps}\omega)}\big),
$$
where $o_\delta(1)$ tends to zero as $\delta\to0$.
Due to the a priory estimates for $Du^\eps$ and by the Birkhoff theorem, the last term on the right-hand side does not exceed
$o_\delta(1)$ for sufficiently small $\eps$.  Applying again the Birkhof theorem we conclude that the first term on the right-hand side  converges to the integral
$$
\int\limits_{Q_T}\int\limits_\Omega(1-p(\omega))|U^\gamma(t,x,\omega)|^{p(\omega)}dxdtd{\bf P}.
$$
Since $p(\cdot)|U^\gamma|^{p(\cdot)}U^\gamma$ can be used as a test function in the definition of two-scale convergence,
the second term on the right-hand side converges to the integral
$$
\int\limits_{Q_T}\int\limits_\Omega p(\omega)|U^\gamma(t,x,\omega)|^{p(\omega)-2}
U^\gamma(t,x,\omega)U(t,x,\omega)\,dxdtd{\bf P}
$$
Summarizing the above relations yields
$$
\liminf\limits_{\eps\to0}\int\limits_{Q_T} |Du^\eps(t,x)|^{p(\tau_{\frac x\eps}\omega)}\,dxdt\geq
\int\limits_{Q_T}\int\limits_\Omega(1-p(\omega))|U^\gamma(t,x,\omega)|^{p(\omega)}dxdtd{\bf P}.
$$
$$
+\int\limits_{Q_T}\int\limits_\Omega p(\omega)|U^\gamma(t,x,\omega)|^{p(\omega)-2}
U^\gamma(t,x,\omega)U(t,x,\omega)\,dxdtd{\bf P}+ o_\delta(1).
$$
Sending first $\delta\to0$ and choosing sufficiently small $\gamma>0$ we conclude that
$$
\int\limits_{Q_T}\int\limits_\Omega|U^\gamma(t,x,\omega)|^{p(\omega)}dxdtd{\bf P}\leq C
$$
with a constant $C$ that does not depend on $\gamma$. By the Fatou lemma this yields the desired
statement. Moreover, we have
$$
\liminf\limits_{\eps\to0}\int\limits_{Q_T} |Du^\eps(t,x)|^{p(\tau_{\frac x\eps}\omega)}\,dxdt\geq
\int\limits_{Q_T}\int\limits_\Omega|U(t,x,\omega)|^{p(\omega)}dxdtd{\bf P}.
$$
\end{proof}

\bigskip
The last Lemma implies that
\begin{equation}\label{elements}
Du\in L^f(Q_T),\quad u\in X^f(Q_T), \quad z_0\in L^{f^*}(Q_T).
\end{equation}
Indeed, by Lemma \ref{l_sol-2sca},
$$
\int_{Q_T}f(Du)\,dxdt= \int_{Q_T}\Big(\min\limits_{w\in \mathcal{G}(\Omega)}\,\int_\Omega|Du(x,t)+w(\omega)|^{p(\omega)}\,d\mP(\omega)\Big) dxdt
$$
$$
\leq \int_{Q_T}\Big(\int_\Omega|Du(x,t)+u_1(x,t,\omega)|^{p(\omega)}\,d\mP(\omega)\Big) dxdt<\infty.
$$
Similarly,
$$
\int_{Q_T}f^*(z_0)\,dxdt\leq \int_{Q_T}\int_\Omega|z(x,t,\omega)|^{p'(\omega)}\,d\mP(\omega)\Big) dxdt<\infty.
$$
It also follows from Lemma  \ref{l_sol-2sca} that
\begin{equation}\label{useful1}
\int_{t_1}^{t_2}\!\int_Gz_0 \cdot Du\,dxdt=\int_{t_1}^{t_2}\!\!\int_G\!\int_\Omega z (Du+u_1)\,dxdtd\mP(\omega).
\end{equation}
Our next goal is to pass to the limit in the viscous term in (\ref{n-s_ini}). To this end we take the difference between the relations
of items (i) and (iii) of Section \ref{ss_func-sp}. The resulting relation reads
$$
\begin{array}{l}
\displaystyle
\int\limits_{t_0}^{t_1}\int\limits_GA\Big(\frac{x}{\eps},Du^\eps\Big)\cdot Du^\eps\,dxdt
\leq
\int\limits_{t_0}^{t_1}\int\limits_G\Big[A\Big(\frac{x}{\eps},Du^\eps\Big)-u^\eps\otimes u^\eps\Big]\cdot\nabla\eta\,dxdt\\[4mm]
\displaystyle
-\int\limits_G\Big(\big[\frac{1}{2}|u^\eps(x, t_1)|^2-u^\eps(x,t_1)\cdot\eta(x)\big]-
\big[\frac{1}{2}|u^\eps(x, t_0)|^2-u^\eps(x,t_0)\cdot\eta(x)\big]\Big)\,dx.
\end{array}
$$
for any $\eta\in C_{0,{\rm sol}}^\infty(G)$.
Considering the relation
$$
\int_G\Big(\frac{1}{2}|u^\eps|^2-u^\eps\cdot\eta\Big)\,dx\Big|_{t=t_0}^{t_1}=\frac{1}{2}\int_G(|u^\eps-\eta|^2)\,dx\Big|_{t=t_0}^{t_1}
$$
and the symmetry of matrices $A$ and $u^\eps\otimes u^\eps$, we derive
$$
\begin{array}{rl}
\displaystyle
\int\limits_{t_0}^{t_1}\int\limits_GA\Big(\frac{x}{\eps},Du^\eps\Big)\cdot Du^\eps\,dxdt
\!&\!\leq \displaystyle
\int\limits_{t_0}^{t_1}\int\limits_G\Big[A\Big(\frac{x}{\eps},Du^\eps\Big)-u^\eps\otimes u^\eps\Big]\cdot D\eta\,dxdt\\[6mm]
&\displaystyle
+\frac{1}{2}\int_G|u^\eps(x, t_0)-\eta(x)|^2\,dx.
\end{array}
$$
Choosing $t_0$ in such a way that $u^\eps(\cdot,t_0)$ converges to $u(\cdot,t_0)$ in $L^2(G)$ and $u(\cdot,t_0)\in W_0^{1,\alpha}(G)$, and passing to
the limit $\eps\to 0$ yields
\begin{equation}\label{1st_limrel}
\begin{array}{rl}
\displaystyle
\lim\limits_{\eps\to0}\ \int\limits_{t_0}^{t_1}\int\limits_GA\Big(\frac{x}{\eps},Du^\eps\Big)\cdot Du^\eps\,dxdt
\!&\!\leq \displaystyle
\int\limits_{t_0}^{t_1}\int\limits_G[z_0-u\otimes u]\cdot D\eta\,dxdt\\[6mm]
&\displaystyle
+\frac{1}{2}\int_G|u(x, t_0)-\eta(x)|^2\,dx.
\end{array}
\end{equation}
We are going to show that $\eta=u(x,t_0)$ can be chosen as a test function in  the last inequality.
Let $\{\eta\big._N\}\big._{N=1}^\infty$ be a sequence of functions $\eta_N\in C_{0,{\rm sol}}^\infty(G)$ such that
$\eta_N\to\eta$ in $W^{1,\alpha}_0(G)$. We substitute $\eta_N$ for a test function in (\ref{1st_limrel}) and pass
to the limit, as $N\to\infty$.  It is clear that $\eta_N\to u(\cdot,t_0)$ in $L^2(G)$. Therefore, the last term on the right-hand side tends to zero.

Regarding the convection term by Lemma \ref{l_emb_imp} we have
$$
|u\otimes u|\in L^1((0,T),L^{\alpha}).
$$
Then
$$
\int\limits_{t_0}^{t_1}\int\limits_G(u\otimes u)\cdot D\eta_N\,dxdt=\int\limits_G\int\limits_{t_0}^{t_1}(u\otimes u)\,dt\cdot
D\eta_N\,dx\to \int\limits_G\int\limits_{t_0}^{t_1}(u\otimes u)\cdot Du(x,t_0)\,dxdt.
$$
By Lemma \ref{l_dens-sol}, the space $C_{0,{\rm sol}}^\infty(G)$ is dense in $W^{1,f}_0(G)$. Therefore, we can assume that $\eta_N$ converges to $u(\cdot,t_0)$ in $W^{1,f}_0(G)$. This yields
$$
\int\limits_{t_0}^{t_1}\int\limits_G z_0\cdot D\eta_N\,dxdt \longrightarrow \int\limits_{t_0}^{t_1}\int\limits_G z_0\cdot Du(x,t_0)\,dxdt ;
$$
here we used the fact that
$$
\int_{t_0}^{^t_1}z_0\,dt \,\in\, L^{f^*}(G).
$$
Letting $t_1=t_0+h$ and combining (\ref{1st_limrel}) with the above limit relations,  we obtain
$$
\begin{array}{rl}
\displaystyle
\lim\limits_{\eps\to 0}\frac{1}{h}\int\limits_{t_0}^{t_1}\int\limits_G A\Big(\frac{x}{\eps},Du^\eps\Big)\cdot Du^\eps\,dxdt
\!\!&\!\displaystyle
\leq
\frac{1}{h}\int\limits_{t_0}^{t_1}\int\limits_G[z_0-u\otimes u]\cdot Du(x,t_0)\,dxdt\\[5mm]
&\displaystyle
=\frac{1}{h}\int\limits_{t_0}^{t_1}\int\limits_Gz_0\cdot Du\,dxdt -R(h)
\end{array}
$$
with
$$
R(h)=\frac{1}{h}\int\limits_{t_0}^{t_0+h}\int\limits_G z_0\cdot(Du(x,t)-Du(x,t_0))\,dxdt-
\int\limits_G\frac{1}{h}\int\limits_{t_0}^{t_0+h}u\otimes u\,dt\cdot Du(x,t_0)\,dx.
$$
With the help of  (\ref{useful1}) we rearrange the last inequality as follows
\begin{equation}\label{vazhno}
\lim\limits_{\eps\to 0}\frac{1}{h}\int\limits_{t_0}^{t_1}\!\int\limits_G A\Big(\frac{x}{\eps},Du^\eps\Big)\cdot Du^\eps\,dxdt
\leq\frac{1}{h}\int\limits_{t_0}^{t_1}\!\int\limits_G\!\int\limits_\Omega \!z\cdot (Du+u_1)\,d\mP(\omega)dxdt- R(h).
\end{equation}
Due to monotonicity of $\mA(\omega,\xi)$,  for any $\Phi\in C_0^\infty(G, \mathcal{D}^\infty(\Omega,\mathbb R^{d(d+1)/2}))$ we have
$$
\frac{1}{h}\int\limits_{t_0}^{t_1}\int\limits_G \Big[A\Big(\frac{x}{\eps},Du^\eps\Big)-A\Big(\frac{x}{\eps},\Phi(x,\tau_{x/\eps}\omega)\Big)\Big]
\cdot [Du^\eps-\Phi(x,\tau_{x/\eps}\omega)]\,dxdt\geq0.
$$
We pass to the limit, as $\eps\to0$, in this relation. The term with the integrand $A(\eps^{-1}x,Du^\eps)\cdot Du^\eps$ has been estimated in (\ref{vazhno}).
In other three terms we pass to the two-scale limit. This yields
$$
\frac{1}{h}\int\limits_{t_0}^{t_1}\int\limits_G\int\limits_\Omega [z-\mA(\omega,\Phi(x,\omega))]\cdot[Du+u_1-\Phi(x,\omega)]\,d\mP(\omega)dxdt
\geq    R(h).
$$
For an arbitrary Lebesgue point $t_0$ of  functions $z(\cdot,t,\cdot)\cdot Du(\cdot, t)$ and $z(\cdot,t,\cdot)\cdot u_1(\cdot, t,\cdot)$ the left-hand side
of the last inequality converges as $f\to 0$ to the following integral
$$
\int\limits_G\int\limits_\Omega[z(x,t_0,\omega)-\mA(\omega,\Phi(x,\omega))]\cdot[Du(x,t_0)+u_1(x,t_0,\omega)-\Phi(x,\omega)]\,d\mP(\omega)dx.
$$
It is also easy to check that both integrals in the definition of $R(h)$ tend to zero, as $h\to 0$. Therefore,
$$
\int\limits_G\int\limits_\Omega[z(x,t_0,\omega)-\mA(\omega,\Phi(x,\omega))]\cdot[Du(x,t_0)+u_1(x,t_0,\omega)-\Phi(x,\omega)]\,d\mP(\omega)dx\geq 0
$$
for any test function $\Phi$. By the standard Minty's arguments
$$
z(x,t_0,\omega)=\mA\big(\omega,Du(x,t_0)+u_1(x,t_0,\omega)\big).
$$
By Lemma \ref{l_sol-2sca}  we have $z(x,t_0,\cdot)\in \mathcal{G}^\perp(\Omega)$. Therefore,
$$
\int_\Omega\mA\big(\omega,Du(x,t_0)+u_1(x,t_0,\omega)\big)\cdot v(\omega)\,d\mP(\omega)=0
$$
for any $v\in \mathcal{G}(\Omega)$, and thus $u_1(x,t_0,\omega)$ is a solution of problem (\ref{cellpbm}) with
$\xi=Du(x,t_0)$. We then conclude that
$$
z_0(x,t_0)=\int_\Omega z(x,t_0,\omega)(\omega)\,d\mP(\omega)=\int_\Omega \mA\big(\omega,Du(x,t_0)+u_1(x,t_0,\omega)\big)\,d\mP(\omega)
$$
$$
=A^{\rm eff}(Du(x,t_0)).
$$
This completes the proof of Theorem \ref{t_main-homo}.

\section{Examples}

In this section we consider examples of random diffusion tensors $A(x,\xi)$.

\bigskip\noindent
{{\bf Example 1. \ \   Voronoi-Poisson tesselation model}.

\medskip\noindent
Consider a Poisson point process in $\mathbb R^d$ with intensity $1$, and construct the Voronoi tessellation (diagram) for this point process. It is known (see \cite{DaVe}) that a.s. the said Voronoi tessellation consists of a countable number of convex polytopes,
we denote them $H_1,\,H_2,\ldots$.  Moreover, the polytopes can be enumerated in such a way that the characteristic function
${\bf 1}_{H_j}(y)$ is a $\mathcal{B}\times\mathcal{F}$-measurable  function of $y$ and $\omega$ for any $j=1,2,\ldots$

Let $\eta_1,\, \eta_2\,\ldots$ be a family of i.i.d. random variable taking on values in $[\alpha,\beta]$ with
$\alpha_0(d)\leq\alpha<\beta\leq\alpha^*$.   We then set
$$
{\tt p}(y)=\sum\limits_{j=1}^\infty \eta_j{\bf 1}_{H_j}(y),\qquad A(y,\xi) = |\xi|^{{\tt p}(y)-1}, \ \
\xi\in\mathbb R^{\frac{d(d+1)}{2}}.
$$
This diffusion matrix $A=A(y,\xi)$ satisfies all the conditions of Theorem \ref{t_main-homo}.

Let $(\Omega,\mathcal{F},{\bf P})$ be the underlying probability space with an ergodic  dynamical system $\tau_y$ such that
${\tt p}(y)=p(\tau_y\omega)$ with  $p(\omega)={\tt p}(0)$.
Problem \eqref{cellpbm} then takes the form: \ \ find
$v\in \mathcal{G}(\Omega)$ such that
$$
\int\limits_\Omega |\xi+v(\omega)|^{p(\omega)-1}\theta(\omega)\,d{\bf P}
$$
for any $\theta\in \mathcal{G}(\Omega)$.

\bigskip\noindent
Taking the convolution of ${\tt p}$ with a $C_0^\infty(\mathbb R^d)$ even function $\varphi=\varphi(y)$ such that $\varphi\geq 0$
and $\int_{\mathbb R^d}\varphi\,dy=1$, denoting the obtained function by $\widehat{\tt p}$ and letting
$A(\omega,\xi)=|\xi|^{\widehat p(\omega)}, \ \ \widehat p(\omega)=\widehat{\tt p}(0)$, we define a diffusion matrix
$A(\omega,\xi)$ with a.s. continuous in $y$ realizations.

\bigskip\noindent
{\bf Example 2. Bernoulli percolation model}.

\medskip\noindent
Consider a checker board in $\mathbb R^d$ with the cell $[0,1)^d$. We associate to each cell a random variable that takes on the value $1$ with probability $q$ and the value $0$ with probability $1-q$, and assume that these random variables are i.i.d.
We denote these random variables by $\zeta_j$, $j\in \mathbb Z^d$, and the corresponding cells by $Q_j$, so that
$Q_j=[0,1)^d+j$. It is known (see \cite{Kes}) that there is a ${\bf p}_{\rm cr}$, $0< {\bf p}_{\rm cr}<1$, such that for $q>{\bf p}_{\rm cr}$ the set $\{\bigcup\limits_j Q_j\,:\, \zeta_j=1\}$ a.s. has a unique unbounded connected component, the so-called infinite cluster. We denote it by $\mathcal{C}$, and introduce the following two random functions:
$$
{\tt p}(y)=\alpha +(\beta-\alpha){\bf 1}_{\mathcal C}(y),\qquad a(y)=1+\sum\limits_{j\in\mathbb Z^d}\zeta_j{\bf 1}_{Q_j}(y)
$$
with $\alpha_0(d)\leq\alpha<\beta\leq\alpha^*$. Then
$$
A(y,\xi)=a(y)|\xi|^{\tt p(y)}, \quad\xi\in \mathbb R^{d(d+1)/2},
$$
is an admissible diffusion matrix.

\bigskip\noindent
{\bf Acknowledgements.}
The essential part of this work was done at A. and N. Stoletov Vladimir State University.
The work of both authors was supported by the Russian Science (Research?) Fund, Project 14-11-00398.

\end{document}